\documentclass{article}

\usepackage{amsfonts}
\usepackage{graphicx}
\usepackage{psfrag}
\usepackage{overpic,color}

\newtheorem{dummy}{Definition}[section]
\newtheorem{lemma}[dummy]{Lemma}
\newtheorem{proposition}[dummy]{Proposition}

\newenvironment{definition}{\begin{dummy}\rm}{\end{dummy}}
\newenvironment{proof}{\addvspace{\bigskipamount}\noindent{\bf Proof\ }}{\par\addvspace{\bigskipamount}}
\newenvironment{remark}{\addvspace{\bigskipamount}\noindent{\bf Remark\ }}{\par\addvspace{\bigskipamount}}
\newenvironment{remarks}{\addvspace{\bigskipamount}\noindent{\bf Remarks\ }}{\par\addvspace{\bigskipamount}}
\newenvironment{theorem}{\addvspace{\bigskipamount}\noindent{\bf Theorem\ }}{\par\addvspace{\bigskipamount}}

\newcommand{\D}{\mathbb{D}}
\newcommand{\N}{\mathbb{N}}
\renewcommand{\S}{\mathbb{S}}
\newcommand{\R}{\mathbb{R}}
\newcommand{\Z}{\mathbb{Z}}

\newcommand{\card}{\mathop{\rm card}\nolimits}
\newcommand{\cl}{\mathop{\rm cl}\nolimits}
\renewcommand{\dim}{\mathop{\rm dim}\nolimits}
\newcommand{\grad}{\mathop{\rm grad}\nolimits}
\newcommand{\im}{\mathop{\rm im}\nolimits}
\renewcommand{\int}{\mathop{\rm int}\nolimits}
\newcommand{\lk}{\mathop{\rm lk}\nolimits}
\renewcommand{\max}{\mathop{\rm max}\nolimits}
\renewcommand{\min}{\mathop{\rm min}\nolimits}

\begin{document}

\title{Another way of answering\\Henri Poincar\'es fundamental question}
\author{Peter Mani-Levitska\\
\\
{\normalsize Riedstrasse~36}\\
{\normalsize 3626 H\"unibach/BE}\\
{\normalsize Switzerland}}
\date{June 10, 2009}
\maketitle

\bigskip

\begin{abstract}
After G.~Perelman's solution of the Poincar\'e Conjecture, this is a different way toward it. Given a simply connected, closed 3-manifold $M$, we produce a homotopy disc $H$, which arises from $M$ by a finite sequence of simple modifications and, almost miraculously, can be imbedded into the ordinary space $\R^{3}$. It follows that $H$ is a disc, hence $M$ is a sphere. In order to construct $H$, we use a special stratification of $M$, based on the fact that $M$ is simply connected.
\end{abstract}

\bigskip

\section{Introduction}

I have been reluctant, after G.~Perelman's work, to present this second way of answering Henri Poincar\'e's question about 3-dimensional homotopy spheres, \cite{milnor03}, \cite{perelman02}, \cite{perelman03g}, \cite{perelman03c}. It probably originates in my misunderstanding of an idea described by S.~Smale in \cite{smale90}: ``First triangulate the 3-manifold, and remove one 3-dimensional simplex. It is sufficient to show the remaining manifold is homeomorphic to a 3-simplex. Then remove one 3-simplex at a time. This process doesn't change the homeomorphism type, and finally one is left with a single 3-simplex.''

I thought that S.~Smale was trying to construct a shelling. Since I was aware of non-shellable spheres \cite{DK78}, \cite{gruenbaum67} \cite{kuperberg01}, \cite{ziegler95}, it seemed natural to relax the conditions a little: replace triangulations by more general cell decompositions, allow to remove more than one 3-cell, each time, and also allow to add certain collections of 3-cells. These combinatorial homotopy operations, together with subdivisions, can be used to transform any 3-dimensional homotopy sphere into some homotopy disc $D$ with a wonderful property: it can be imbedded into the ordinary space $\R^{3}$. Now the generalized Sch\"onflies theorem allows us to conclude that $D$ is homeomorphic to the standard 3-disc, and we have arrived.

I sincerely thank my wife Anna. She has always supported my work, even when it looked like nothing. She also opened me the door to the miraculous world of melody, harmony and rhythm. I also thank Stephan Fischli who is patiently teaching me the art of writing.

This essay is devoted to the meditation of God's Word among the inhabitants of the earth.

\section{Basic notions}

We shall quote freely from \cite{BC70}, \cite{BJ73}, \cite{gruenbaum67}, \cite{milnor63}, \cite{munkres63}, \cite{munkres99}, \cite{spanier66} and \cite{ziegler95}. Let us elaborate a little on two areas.

\subsection{Differential geometry}

Many of our arguments rest on the fact \cite{bing83} that every compact 3-manifold has a smooth atlas, where ``smooth'' here means ``of class $C^1$''. A smooth manifold is called an $n$-sphere, if it is diffeomorphic to the standard sphere $\S^n=\{x\in\R^{n+1}:\,|x|=1\}$, and an $n$-disc, if it is diffeomorphic to $\D^n=\{x \in\R^n:\,|x|\le 1\}$.

If $M$ and $N$ are smooth manifolds, then $f:X\rightarrow Y$, with $X\subset M$, $Y\subset N$, is a smooth imbedding, if there exist open sets $U\subset M$, $V\subset N$, as well as a diffeomorphim $g:U\rightarrow V$, such that $X\subset U$, $Y\subset V$ and $g|X=f$.

We write $\dim M$, for the dimension of a manifold $M$, $\partial M$ for its boundary and $\int M$ for its interior. By a smooth deformation of $M$ we understand a smooth map $H:M\times [0,1]\rightarrow M$, such that $H_\tau:M\rightarrow M$ is a diffeomorphism, for every $\tau\in[0,1]$, with $H_0={\rm Id}_M$. Here we adopt the notation $H_\tau(x)=H(x,\tau)$.

A Morse function $\varphi:M\rightarrow\R$ is said to be climbing, of length $k$, if it has $k+1$ critical points $p_0,\ldots,p_k$, and if $\varphi(p_0)<\varphi(p_1)<\ldots<\varphi(p_k)$. $(p_0,\ldots,p_k)$ is its chain of critical points.

Whenever we deal with a chart $(U,h)$ of a manifold $M$, we assume that the environment $U$ is compact.

\subsection{Polytopes and cell complexes}

If $P$ is a (convex) polytope, we denote by ${\cal F}P$ the collection of all its faces, including $\emptyset$ and $P$, and by $\Delta^iP=\{X\in{\cal F}P:\ \dim X=i\}$ the set of its $i$-dimensional faces. The elements of $\Delta^0 P$ are the vertices, and those of $\Delta^1P$ are the edges of $P$. In the case $\dim P=3$, we often use the term ``face'' for a member of $\Delta^2P$.

We also need some notations from the $\Z_2$-homology of a simplicial complex $C$. Remember that the $\Z_2$-chains in $C$ can be identified with their supports, so that an $i$-chain in $C$ appears as the union of some set $A\subset\Delta^iC=\{X\in C:\ \dim X=i\}$. Compare \cite{SF34} for a careful description.

We often work in the piecewise linear category, where the terms ``space'' and ``homeomorphism'' mean ``polyhedron'' and ``piecewise linear homeomorphism'' \cite{RS72}. Sometimes, however, it is more convenient to consider the piecewise smooth category, whose maps can be decomposed into finitely many diffeomorphisms.

Whenever we deal with a closed interval $[a,b]\subset\R$, let us assume that $a<b$.

\section{Homotopy spheres, Morse functions, and stra\-tifications}

\begin{definition}\label{def:homotopy disc}
A {\em homotopy 3-sphere} is a closed, compact, simply connected 3-manifold. A {\em homotopy 3-disc} is a compact, simply connected 3-manifold $M$, whose boundary $\partial M$ is a sphere.
\end{definition}

\begin{definition}\label{def:stratification}
By a {\em stratification} of some 3-manifold $M$ we understand a sequence of submanifolds $S_1,\ldots,S_r$, $r\in\N$, with $M=S_1\cup\ldots\cup S_r$, such that there exist homeomorphisms $\ell_i:S_i\rightarrow F_i\times[0,1]$, $1\le i\le r$, where the $F_i$ are orientable 2-manifolds, and where the equations $S_i\cap S_{i+1}=U(S_i,\ell_i)\cap L(S_{i+1},\ell_{i+1})$ and $S_i\cap S_j=\emptyset$, in the case $|i-j|\ge 2$, always hold.

Here, if $\ell:A\rightarrow B\times[0,1]$ is a homeomorphism, we write $U(A,\ell)=\ell^{-1}(B\times\{1\})$ and $L(A,\ell)=\ell^{-1}(B\times\{0\})$ for the upper and lower boundaries of $A$, with respect to $\ell$. The spaces $S_i$ are called the {\em strata} in $M$.
\end{definition}

\begin{remarks}
\begin{itemize}
\item[(1)]
If the manifolds $F_i$ are all contained in the plane $\R^2$, then $S_1,\ldots,S_r$ is called a {\em planar stratification} of $M$.
\item[(2)]
Given two stratifications $\rho=R_1,\ldots,R_t$, and $\sigma=S_1,\ldots,S_r$ of $M$, we say that $\sigma$ is a {\em refinement} of $\rho$, if every space $S_i$ is contained in some $R_j$.
\end{itemize}
\end{remarks}

\begin{definition}\label{def:canonical chart}
Consider a compact 3-manifold $M$, a Morse function $\varphi:M\rightarrow\R$ and a critical point $p$ of $\varphi$. Let $n$ be the index of $\varphi$ at $p$, and choose a number $\delta>0$. A chart $\alpha=(U,h)$ of $M$, with $V=h(U)\subset\R^3$, is called {\em canonical}, with respect to $M$, $\varphi$, $p$ and $\delta$, if it satisfies
\begin{itemize}
\item[(1)]
$p\in\int U$, $h(p)=0$
\item[(2)]
the map $\psi:V\rightarrow\R$, given by $\psi(x)=\frac{1}{\delta}(\varphi\circ h^{-1}(x)-\varphi(p))$, has the form $\psi(x)=\sum_{i=1}^3\varepsilon_i x_i^2$, with $\varepsilon_i=1$ for $i\le 3-n$ and $\varepsilon_i=-1$, otherwise.
\item[(3)]
$V$ is the standard disc $\D^3$, if $n\in\{0,3\}$.
\item[(4)]
In the case $n=1$, $V$ is the closure of the bounded component of $\R^3\setminus S$, where we now proceed to describe the 2-sphere $S$.
\end{itemize}
Consider the 2-discs $H^+=\{x\in\R^3:\ \psi(x)=-1,\ 0\le x_3\le 2\}$ and $H^-=\{x\in\R^3:\ \psi(x)=-1,\ -2\le x_3\le 0\}$, and note that $L=\{x\in\R^3:\ \psi(x)=1\}$ is a cylindrical surface. Next, let us associate to each $\varepsilon\in\{+,-\}$ the ring $R^\varepsilon=\{x\in\R^3:$ there exist a number $\mu>0$ and a solution $c:[0,\mu]\to\R^3$ of the differential equation $y'=(\grad\psi)(y)$ with $c(0)\in\partial H^\varepsilon$, $c(\mu)\in L$, and $x\in\im(c)\}$, and denote by $Q$ the closure of the bounded component of $L\setminus (R^+\cup R^-)$. Now the sphere $S$, mentioned above, is the space $S=H^+\cup H^-\cup R^+\cup R^-\cup Q$.

\begin{figure}[ht]
\centering
\begin{overpic}{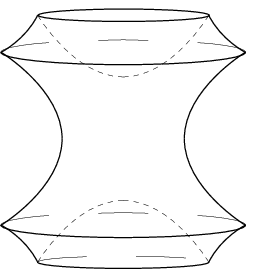}
\put(42,58){$H^+$}
\put(42,32){$H^-$}
\put(88,92){$R^+$}
\put(88,-2){$R^-$}
\put(75,45){$Q$}
\end{overpic}
\end{figure}

It follows from the generalized Sch\"onflies theorem \cite{bing83},\cite{brown60} that $V$ is a 3-disc, but in view of the rotational symmetry of $S$, there also exists an elementary argument. Let us say that $R^+$ and $R^-$ are the top and the bottom ring of $V$, $H^+$ and $H^-$ are the top and the bottom bay, while $Q$ is the cloak.
\begin{itemize}
\item[(5)]
In the case $n=2$, $V$ arises from the sphere described below (4) by a rotation around the $y$-axis, with an angle of $\frac{\pi}{2}$. The subsets of $\partial V$, corresponding to $R^+$, $R^-$, $H^+$, $H^-$ and $Q$ are now called the right and left ring, the right and left bay, and the cloak.
\end{itemize}
\end{definition}

\begin{proposition}\label{prop:stratification}
Let $M$ be a compact 3-manifold, $\varphi:M\rightarrow\R$ a climbing Morse function, and $(p_0,\ldots,p_k)$ its chain of critical points. Consider a sequence of positive numbers $(\delta_0,\ldots,\delta_k)$, which satisfy $\varphi(p_i)+\delta_i<\varphi(p_{i+1})-\delta_{i+1}$, always, and choose canonical charts $\alpha_i=(U_i,h_i)$ with respect to $M$, $\varphi$, $p_i$ and $\delta_i$, $0\le i\le k$. There exists a stratification of the manifold $N=\cl(M\setminus(U_0\cup\ldots\cup U_k))$.
\end{proposition}

\begin{proof}
Let us show, by induction on $i$ that
\begin{itemize}
\item[(1)]
there exists a stratification of $N_i=\{x\in N:\ \varphi(x)\le\varphi(p_i)-\delta_i\}$ for every $0\le i\le k$.
\end{itemize}
Note that $N_0=\emptyset$ and $N_k=N$. We remember a basic fact in differential geometry \cite{kosinski93}, \cite{milnor63}:
\begin{itemize}
\item[(2)]
Given a compact Riemann manifold $W$, a smooth function $f:W\rightarrow\R$ and real numbers $\rho<\sigma$, which satisfy $(\grad f)(x)\ne0$ for every $x\in f^{-1}[\rho,\sigma]$, we find a diffeomorphism $g:f^{-1}(\rho)\times[\rho,\sigma]\rightarrow f^{-1}[\rho,\sigma]$, such that $g(f^{-1}(\rho)\times\{\tau\})=f^{-1}(\tau)$, always.
\end{itemize}
Now, with $\rho=\max(\varphi(U_0))$ and $\sigma=\min(\varphi(U_1))$, we obtain a diffeomorphism $g:\varphi^{-1}(\rho)\times[\rho,\sigma]\rightarrow\varphi^{-1}[\rho,\sigma]$ such that $g(\varphi^{-1}(\rho)\times\{\tau\})=\varphi^{-1}(\tau)$, for every $\tau\in[\rho,\sigma]$. Associate to every $x\in\partial U_0$ the curve $Z(x)=g(\{x\}\times[\rho,\sigma])$ and observe, in view of (2), that $S_1=\bigcup\{Z(x):\ x\in\partial U_0\}$ is a stratification of $N_1$.

For the inductive step, let us assume that we have constructed a stratification $S_1,\ldots,S_r$ of $N_i$, for some integer $i\ge1$. We begin with the situation where
\begin{itemize}
\item[(3)]
the index $n$ of $\varphi$ at $p_i$ is 1.
\end{itemize}
We write $(U,h)=(U_i,h_i)$, and follow the notation of Definition~\ref{def:canonical chart}. Our construction of the canonical charts implies that the map $\psi$ satisfies $(\grad\psi)(x)\in T_x(\partial V)$, for every $x\in R^+\cup R^-$, where $T_x(\partial V)$ stands for the tangent space of $\partial V$ at $x$. Now we choose a Riemann metric $g$ on $M$, such that $g_x(u,v)=\langle h'(u),h'(v)\rangle$, everywhere in some neighbourhood of $U=h^{-1}(V)$.

The gradient of $\varphi$, with respect to $g$, fulfills the equation $(\grad\varphi)(x)=(h')^{-1}((\grad\psi)(h(x)))$, at every point $x$ in this neighbourhood, where $\grad\psi$ is taken with respect to the standard Euclidean metric. Consequently, $(\grad\varphi)$ is a smooth vector field on $W=\{x\in N:\ \mu\le\varphi(x)\le\nu\}$, and $(\grad\varphi)(x)$ lies in the tangent space $T_x(h^{-1}(R^+\cup R^-))$, whenever $x$ belongs to $\partial W\cap h^{-1}(R^+\cup R^-)$. Here, we have written $\mu=\min(\varphi(U))$ and $\nu=\max(\varphi(U))$.

Associate to every $x\in W$ with $\varphi(x)=\mu$ the trajectory $a_x:[0,\delta(x)]\rightarrow W$ with $a_x(0)=x$, where the number $\delta(x)$ is chosen such that $\varphi(a_x(\delta(x)))=\nu$. We write $S_{r+1}=\bigcup\{\im(a_x):\ x\in W,\ \varphi(x)=\mu\}$ and observe that $S_1,\ldots,S_r,S_{r+1}$ is a stratification of $\{x\in N:\ \varphi(x)\le\nu\}$.

By repeating the construction below (2) once more, we obtain a stratification of $N_{i+1}$. The inductive step is completed, under the assumption (3).
\begin{itemize}
\item[(4)]
Assume that the index of $\varphi$ at $p_i$ is 2.
\end{itemize}
We proceed as we did under (3), with one little difference: $N_i$ satisfies $h^{-1}(Q)\subset\partial N_i$ instead of $h^{-1}(H^+\cup H^-)\subset\partial N_i$.
\begin{itemize}
\item[(5)]
Assume that the index of $\varphi$ at $p_i$ is 0 or 3.
\end{itemize}
Here we can use a trivial argument, similar to the passage below (2). Proposition~\ref{prop:stratification} is established.
\end{proof}

\section{Splitting discs}

\begin{definition}\label{def:splitting disc}
Let $M$ be a homotopy 3-disc. By a {\em splitting disc} in $M$ we understand a 2-disc $D\subset M$, transverse to the boundary $\partial M$, and satisfying $\int D\subset\int M$, $\partial D\subset\partial M$.
\end{definition}

\begin{definition}\label{def:linking number}
Consider a homotopy 3-disc $M$ and two piecewise smooth im\-beddings $a:\S^1\rightarrow M$, $b:\S^1\rightarrow M$ such that $\im(a)\cap\im(b)=\emptyset$. Choose a triangulation $\tau=(C,h)$ of $M$, such that
\begin{itemize}
\item[(1)]
$h(A)=\im(a)$, $h(B)=\im(b)$, for two $\Z_2$ 1-cycles $A$, $B$ in the simplicial complex $C$
\item[(2)]
there exists a 2-chain $D$ in $C$, transverse to $B$, which satisfies $\partial D=A$.
\end{itemize}
The {\em linking number} $\lk(a,b)$ is the element of $\{0,1\}$ which differs from $\card(B\cap D)$ by an even integer.
\end{definition}

\begin{proposition}\label{prop:linking numbers}{\ }
\begin{itemize}
\item[(1)]
$\lk(a,b)$ does not depend on the choice of $\tau$ and $D$
\item[(2)]
$\lk(a,b)=\lk(b,a)$
\item[(3)]
if $b'$ is homotopic to $b$ in $M\setminus\im(a)$, then $\lk(a,b')=\lk(a,b)$
\end{itemize}
\end{proposition}

\begin{remark}
Compare \cite{bing83} and \cite{RS72}. The book \cite{bing83} contains a proof of Proposition~\ref{prop:linking numbers} in the case $M=\R^3$. This proof can easily be adapted to the present situation.
\end{remark}

\begin{proposition}\label{prop:submanifolds}
Consider a homotopy 3-disc $M$ and a splitting disc $D\subset M$. There exists a unique pair $U$ and $V\ne U$ of submanifolds in $M$, such that
\begin{itemize}
\item[(1)]
$U\cup V=M$, $U\cap V=D$
\item[(2)]
$U$ and $V$ are 3-manifolds, bounded by a 2-sphere.
\end{itemize}
\end{proposition}

\begin{proof}
After a homeomorphism we may assume that there exists a simplicial complex $C$ with $|C|=M$, and a subcomplex $C'$ of $C$, which satisfies $|C'|=D$. Denote by $U_1,\ldots,U_n$ the components of $M\setminus D$. Obviously, we have $n\le2$, so it remains to show that
\begin{itemize}
\item[(1)]
$n=2$.
\end{itemize}
Consider a triangle $X\in\Delta^2C'$ and a point $p\in\int X$. Let $Y_1$ and $Y_2\ne Y_1$ be the tetrahedra in $C$ which contain $X$, and choose points $q_i\in\int Y_i$. If we write $U_i$ for the component of $M\setminus D$ with $q_i\in U_i$, our claim (1) is equivalent to
\begin{itemize}
\item[(2)]
$U_1\ne U_2$.
\end{itemize}
Assuming otherwise, we write $U_1=U_2=U$, and find a piecewise linear path $\beta:[0,1]\rightarrow U\cap\int M$, which satisfies $\beta(0)=q_1$, $\beta(1)=q_2$. We can assume that $B=\im(\beta)\cup[p,q_1]\cup[p,q_2]$ is a circle in $M$. Since $B\cap D=\{p\}$, it follows that $\lk(B,\partial D)=1$. But $M$ is a homotopy disc, therefore $B$ is homotopic to any point $q\in M\setminus D$, hence Proposition~\ref{prop:linking numbers} tells us that $\lk(B,\partial D)=0$. We have reached a contradiction, and (2) follows.

Obviously, $U_1$ and $U_2$ are 3-manifolds, bounded by a 2-sphere, and Proposition~\ref{prop:submanifolds} is established.
\end{proof}

\pagebreak

\section{Imbedding into the 3-dimensional Euclidean space}

\begin{definition}\label{def:homotopy operation}
Consider a 3-manifold $M$ with non-empty boundary and a 3-disc $D\subset M$ such that $\partial D\cap\partial M$ is a 2-disc. We say that $V=\cl(M\setminus D)$ arises from $M$ by a {\em reduction}, and that $M$ arises from $V$ by an {\em extension}. A manifold $N$ arises from $M$ by a {\em homotopy operation}, if there exists a finite sequence $M=M_1,M_2,\ldots,M_r=N$ such that $M_{i+1}$ arises from $M_i$ either by a reduction or an extension, for every $i\in\{1,2,\ldots,r-1\}$.
\end{definition}

\begin{proposition}\label{prop:homotopy operation}{\ }
\begin{itemize}
\item[(1)]
If $M$ is a homotopy 3-disc, and if $N$ arises from $M$ by a homotopy operation, then $N$ is also a homotopy disc.
\item[(2)]
If $M$ is actually a 3-disc, then so is $N$.
\end{itemize}
\end{proposition}

\begin{proof}
Lemma~3.25 in \cite{RS72} confirms that $M$ and $N$ are homeomorphic.
\end{proof}

\begin{proposition}\label{prop:sphere}
Consider 3-discs $A,B$ which satisfy $\partial A=\partial B=A\cap B$. The space $A\cup B$ is a 3-sphere.
\end{proposition}

\begin{proof}
Compare the statement 2.B of chapter IV in \cite{bing83}.
\end{proof}

\begin{definition}\label{def:true companion}
Let $M$ be a homotopy 3-sphere, $\varphi:M\rightarrow\R$ a climbing Morse function, $(p_0,\ldots,p_k)$ its chain of critical points, $(\delta_0,\ldots,\delta_k)$ a sequence of positive numbers, which satisfy $\varphi(p_i)+\delta_i<\varphi(p_{i+1})-\delta_{i+1}$, always. We choose canonical charts $\alpha_i=(U_i,h_i)$ with respect to $M$, $\varphi$, $p_i$ and $\delta_i$, $0\le i\le k$. A homotopy 3-disc $H$, which arises from $M\setminus\int U_0$ by a sequence of homotopy operations, is called a {\em true companion} of $M$, if
\begin{itemize}
\item[(1)]
there exists a stratification of $H$ and
\item[(2)]
$H\cap\int U_i=\emptyset$, for every $i\in\{1,\ldots,k\}$.
\end{itemize}
\end{definition}

Remember Definition~\ref{def:stratification}.

\begin{proposition}\label{prop:true companion}
There exists a true companion of $M$, in the sense of Definition~\ref{def:true companion}.
\end{proposition}

\begin{proof}
We begin with a stratification $S_1,\ldots,S_r$ of $N=\cl(M\setminus(U_0\cup\ldots\cup U_k))$, according to Proposition~\ref{prop:stratification}. Next, let us choose a cell decomposition $C$ of $N$, which respects $S_1,\ldots,S_r$. This means that, for every element $X$ of $\Delta^3C$, there exist a convex polygon $P\subset\R^2$, a homeomorphism $h:P\times[0,1]\rightarrow X$ and a number $i\in\{1,\ldots,r\}$, which satisfy $h(P\times\{0\})\subset L(S_i,\ell_i)$ and $h(P\times\{1\})\subset U(S_i,\ell_i)$, where the maps $\ell_i:S_i\rightarrow F_i\times[0,1]$ have been introduced by Definition~\ref{def:stratification}.

After passing to a refinement $S_1',\ldots,S_s'$ of $S_1,\ldots,S_r$, according to Remark~(2) below Definition~\ref{def:stratification}, and to a cell decomposition $C'$ of $N$, which respects $S_1',\ldots,S_s'$, we can construct pairwise disjoint simple paths $P_1,\ldots,P_k$ in $\Delta^0C'\cup\Delta^1C'$, such that $P_i$ connects $\partial U_0$ to $\partial U_i$.

Finally, we produce a refinement $S_1'',\ldots,S_t''$ of $S_1',\ldots,S_s'$ and a cell decomposition $C''$ of $N$, with subcomplexes $Q_1,\ldots,Q_k$ around the paths $P_1,\ldots,P_k$, and with homeomorphisms $g_i:\D^2\times[0,1]\rightarrow|Q_i|$, which satisfy
\begin{itemize}
\item[(1)]
$\im(g_i)\cap\im(g_j)=\emptyset$ for $i\ne j$
\item[(1)]
$g_i(\D^2\times(0,1))\cap U_j=\emptyset$ for all $j$
\item[(3)]
$g_i(\D^2\times\{0\})\subset\partial U_0$
\item[(4)]
$g_i(\D^2\times\{1\})\subset\partial U_i$
\end{itemize}
$T_i=U_i\cup\im(g_i)$ is a 3-disc in $\cl(M\setminus U_0)$, and $\cl(M\setminus(U_0\cup T_1\cup\ldots\cup T_k))$ a true companion of $M$, as required by Proposition~\ref{prop:true companion}.
\end{proof}

\begin{proposition}\label{prop:disc}
Let $M$ be a homotopy 3-disc, and $C\subset\partial M$ a circle. There exists a splitting 2-disc $D$ in $M$ such that $\partial D=C$.
\end{proposition}

\begin{proof}
We can use the fact that $M$ is simply connected, together with Dehn's Lemma, established by C.D.~Papakyriakopoulos \cite{papakyriakopoulos57}, but there exists an easier argument: $\partial M\setminus C$ has two components, both of which are open 2-discs. Choose one of them and push it inside $M$, while keeping its boundary fixed. The result satisfies the requirements of Proposition~\ref{prop:disc}.
\end{proof}

Our main goal here states that, given any homotopy 3-sphere $M$, we find a true companion of $M$ which can be imbedded into the Euclidean space $\R^3$. The following lemma serves as a first step toward this goal.

\begin{lemma}\label{lemma:circle}
Let $S$ be a connected orientable surface of genus $g\ge1$, and assume that $T\subset S$ is homeomorphic to a compact surface in the plane $\R^2$. There exists a circle $C$ in $S\setminus T$, whose homology class in $H_1(S;\Z_2)$ is not zero.
\end{lemma}

\begin{proof}
We begin with the situation where
\begin{itemize}
\item[(1)]
$\partial S=\emptyset$.
\end{itemize}
The boundary $\partial T$ of $T$ consists of pairwise disjoint circles $K_1,\ldots,K_r$. Let us proceed by double induction, first on $g$, and then on $r$. Denote by $\langle X\rangle\in H_1(S;\Z_2)$ the homology class of the circle $X\subset S$. In the case $(g,r)=(1,1)$, $S$ is a torus and $T\subset S$ a disc, compare \cite{giblin77}. Consider a circle $C\subset S$ with $\langle C\rangle\ne\emptyset$. If $C$ meets $T$, we easily find, by pushing it out of $T$, a circle $C'$, homologuous to $C$ and disjoint to $T$.

For the inductive step, let us begin with the assumption that
\begin{itemize}
\item[(2)]
$\langle K_i\rangle\ne\emptyset$, for some $i\in\{1,\ldots,r\}$.
\end{itemize}
Again, we can push $K_i$ away from $T$, and obtain the required circle. If (2) does not hold, there exist 2-chains $V_i\subset S$ which satisfy $\partial V_i=K_i$, always. In the case $r=1$, the surface $T$ is a disc, and we proceed as we did in the case $(g,r)=(1,1)$.

Now we assume that $r\ge2$, and begin with the situation where
\begin{itemize}
\item[(3)]
$V_i\subset T$, for some $i\in\{1,\ldots,r\}$.
\end{itemize}
Notice that $V_i$ is a planar surface with a single boundary component, and hence a disc, while $T'=T\setminus V_i$ has fewer boundary circles than $T$. The inductive assumption gives us a circle $C\subset S\setminus T'$ with $\langle C\rangle\ne0$, and we can again push $C$ out of $V_i$ in order to obtain the desired result.

Finally, if
\begin{itemize}
\item[(4)]
$V_i\not\subset T$, for every $i\in\{1,\ldots,r\}$,
\end{itemize}
let us look at the closed orientable surface $W_i$, which arises from $S$ by attaching a disc along $K_i$ \cite{kosinski93}. Since $V_i'=\cl(S\setminus V_i)$ is another 2-chain in $S$ with $\partial V_i'=K_i$, we may assume that the genus $g_j$ of $W_j$ is smaller than $g$, for at least one number $j\in\{1,\ldots,r\}$. If $g_j=0$, then $T'=T\cup V_j$ is still a bounded planar surface, and it has fewer than $r$ components. Hence, by the inductive assumption, we obtain a circle $C\subset S\setminus T'\subset S\setminus T$ with $\langle C\rangle\ne0$, and our Lemma follows. In the case $1\le g_j\le g-1$, the inductive assumption produces a circle $C\subset V_j\setminus T$ such that $\langle C\rangle\ne0$ in $H_1(W_j;\Z_2)$. Consequently, $\langle C\rangle\ne0$ in $H_1(S;\Z_2)$. Our proof is complete, under the assumption (1).

If $\partial S$ is not empty, we construct a closed orientable surface $S'$, which arises from $S$ by attaching a 2-disc $D(K)$ along every boundary circle $K$ of $S$, and find a circle $C'$ in $S'\setminus T$, whose homology class in $H_1(S';\Z_2)$ is not zero. As in the beginning of this proof, there arises a circle $C$, homologuous to $C'$ in $S'$, which is still disjoint to $T$, but also disjoint to every disk $D(K)$. Lemma~\ref{lemma:circle} follows.
\end{proof}

\begin{proposition}\label{prop:true companion planar stratification}
With the terminology of Definition~\ref{def:true companion} and Proposition~\ref{prop:true companion}, there exist a true companion $H$ of $M$ with a planar stratification.
\end{proposition}

Remember Remark~(1) below Definition~\ref{def:stratification}.

\begin{proof}
According to Proposition~\ref{prop:true companion}, there exists a true companion $H$ of $M$. We consider a stratification $S_1,\ldots,S_s$ of $H$, with homeomorphisms $\ell_i:S_i\rightarrow F_i\times[0,1]$, as described in Definition~\ref{def:stratification}. Remember that every $F_i$ is an orientable 2-manifold in $\R^3$, and that the equalities $S_i\cap S_{i+1}=U(S_i,\ell_i)\cap L(S_{i+1},\ell_{i+1})$, and $S_i\cap S_j=\emptyset$ in the case $|i-j|\ge2$, always hold.

Now comes what we may call the core of this essay.
\begin{itemize}
\item[(1)]
After some homotopy operations, $H$ will be transformed into a homotopy disc $H'$ which allows a planar stratification $S_1',\ldots,S_t'$. The corresponding homeomorphisms, in the sense of Definition~\ref{def:stratification} shall be denoted by $\ell_i':S_i'\rightarrow F_i'\times[0,1]$. \end{itemize}
Since $L(S_1,\ell_1)$ is a proper subspace of the 2-sphere $\partial H$, we can certainly imbed $F_1$ into the plane $\R^2$. Hence, if (1) does not already hold for $S_1,\ldots,S_s$, we find a number $r\le s$ such that $F_r$ is non-planar, whereas every $F_i$ with $i\le r-1$ is planar.
\begin{itemize}
\item[(2)]
After some extensions, applied to $H$, we may assume that the inclusion $U(S_i,\ell_i)\subset L(S_{i+1},\ell_{i+1})$ holds, for every $i\in\{1,\ldots,r-1\}$.
\end{itemize}
In simpler words, we can say that $S_1\cup\ldots\cup S_r$ looks like an inverted pyramid. If (2) is not correct for $H$ itself, let us write $j=\min\{i\in\{1,\ldots,r-1\}:\ U(S_i,\ell_i)\not\subset L(S_{i+1},\ell_{i+1})\}$, and begin with the case where
\begin{itemize}
\item[(3)]
some component $P$ of $\cl(U(S_j,\ell_j)\setminus L(S_{j+1},\ell_{j+1}))$ is a disc.
\end{itemize}
Above each $X\subset L(S_{j+1},\ell_{j+1})$ stands the tower $T(X)=\ell_{j+1}^{-1}(\ell_{j+1}(X)\times[0,1])$. If $\partial P\cap S_{i+1}\ne\emptyset$, we denote by $F_1,\ldots,F_q$ the components of $\partial P\cap S_{j+1}$, which are all 1-dimensional. Now we attach a 3-cell $W$ to $H$, such that $W\cap H$ is the disc $P\cup T(F_1)\cup\ldots\cup T(F_q)\subset\partial W\cap\partial H$, and extend $S_1,\ldots,S_s$ to a stratification of $H\cup W$.

Next, look at the situation where
\begin{itemize}
\item[(4)]
the assumption (3) does not hold,
\end{itemize}
and choose a component $P$ of $\cl(U(S_j,\ell_j)\setminus L(S_{j+1},\ell_{j+1}))$, together with a 2-disc $D\subset P$, which meets two different components, $A$ and $B$, of $\partial P$, each of them in a line segment, and satisfies $D\setminus(A\cup B)\subset\int P$. Let us also require that $D\cap A$, as well as $D\cap B$, either belong to $L(S_{j+1},\ell_{j+1})$ or are disjoint to it. We associate to each $X\in\{A,B\}$ the tower $T(X)$, as described below (3), if $X$ belongs to $L(S_{j+1},\ell_{j+1})$, and the space $T(X)=X\cap D$, otherwise. Now we attach a 3-cell $W$ to $H$, such that $W\cap H=\partial W\cap\partial H=D\cup T(A)\cup T(B)$, and again extend $S_1,\ldots,S_s$ to a stratification of $H\cup W$.

After a finite number of steps, as described below (3) and (4), we arrive at a situation where (2) is satisfied. Remember, however, that an extension might transform a planar surface $F_i$ into a non-planar surface. This could decrease the number $r$, as described above (2). But, with this in mind, we still have confirmed (2).

Now, Lemma~\ref{lemma:circle} says, that there exists a circle $C$ in $L(S_r,\ell_r)\setminus U(S_{r-1},\ell_{r-1})$, whose homology class in $H_1(L(S_r,\ell_r))$ is not zero. Proposition~\ref{prop:disc} produces a splitting 2-disc $D$ in $H$, bounded by $C$.
\begin{itemize}
\item[(5)]
$D$ can be chosen in such a way that $D\cap S_i=\emptyset$ for every $i\le r-1$.
\end{itemize}
Let us begin with any splitting 2-disc $D$ in $H$ with $\partial D=C$, and write $j=\min\{i\in\{1,\ldots,s\}:\ D\cap S_i\ne\emptyset\}$. Assuming that $j\le r-1$, we know from (2) that $S_j\cap S_{j+1}=U(S_j,\ell_j)$. Hence we can consider the space $N_\delta=S_j\cup\ell_{j+1}^{-1}(\ell_{j+1}(U(S_j,\ell_j))\times[0,\delta])$, for some $\delta\in(0,1)$. There exists a diffeomorphism $f:N_\delta\rightarrow N_\delta$, which leaves the boundary of $N_\delta$ pointwise fixed and carries $D\cap N_\delta$ onto a subset of $N_\delta\setminus S_j$. If we replace $D$ by $(D\setminus N_\delta)\cup f(D\cap  N_\delta)$, we increase the number $j$. After a finite iteration of this process, we obtain a disc $D$ which satisfies (5).

According to Proposition~\ref{prop:submanifolds}, there exists a pair $U$ and $V\ne U$ of 3-manifolds in $H$, such that $U\cup V=H$ and $U\cap V=D$. $U\cap L(S_r,\ell_r)$ is a 2-chain in $L(S_r,\ell_r)$, whose boundary coincides with $C$, contrary to Lemma~\ref{lemma:circle}. The statement (1), and with it Proposition~\ref{prop:true companion planar stratification}, are established.
\end{proof}

\begin{proposition}\label{prop:imbedding}
Consider a homotopy 3-disc $H$. If there exists a planar stratification of $H$, then it can be imbedded into the Euclidean space $\R^3$.
\end{proposition}

\begin{proof}
According to Definition~\ref{def:stratification} and Remark~(1) below it, we can decompose $H$ into submanifolds $S_1,\ldots,S_r$, such that there exist homeomorphisms $\ell_i:S_i\rightarrow F_i\times[0,1]$, $1\le i\le r$, where $F_i$ is a 2-dimensional submanifold of the plane $\R^2$, and where the equalities $S_i\cap S_{i+1}=U(S_i,\ell_i)\cap L(S_{i+1},\ell_{i+1})$, and $S_i\cap S_j=\emptyset$ in the case $|i-j|\ge2$, always hold. We construct a sequence of spaces $M_i=G_i\times[i-1,i]$, $1\le i\le r$, where $G_i$ is a compact 2-manifold in $\R^2$, together with a sequence of homeomorphisms $g_i:S_1\cup\ldots\cup S_i\rightarrow M_1\cup\ldots\cup M_i$, which satisfy $g_i(S_i)=M_i$, always.

Let us proceed by induction on $i$, and set $M_1=F_1\times[0,1]$, $g_1=\ell_1$. Having found $M_i$ and $g_i$, we extend $g_i|U(S_i,\ell_i)$ to a homeomorphism $h$ between $L(S_{i+1},\ell_{i+1})$ and some space $L\subset\R^2\times\{i\}$, containing $G_i\times\{i\}$. Now we define the map $g_{i+1}$ by setting $g_{i+1}(x)=g_i(x)$ for $x\in S_1\cup\ldots\cup S_i$ and $g_{i+1}(\ell_{i+1}^{-1}(x,\tau))=(h(\ell_{i+1}^{-1}(x,0)),i+\tau)$ for $(x,\tau)\in F_{i+1}\times[0,1]$. The homeomorphism $g_r$ imbeds $H$ into the space $\R^3$, as promised.
\end{proof}

\begin{theorem}
Every 3-dimensional homotopy sphere $M$ is homeomorphic to the standard sphere $\S^3$.
\end{theorem}

\begin{proof}
According to Proposition~\ref{prop:sphere} and Definition~\ref{def:true companion}, it is good enough to find a true companion of $M$, which is homeomorphic to the disc $\D^3$. Propositions~\ref{prop:true companion planar stratification} and \ref{prop:imbedding} produce a true companion $H$ of $M$ together with an imbedding $f:H\rightarrow\R^3$. According to the piecewise linear Sch\"onflies theorem for $\R^3$ \cite{bing83}, \cite{graub50}, there exists a piecewise linear homeomorphism $g:\R^3\rightarrow\R^3$, which carries the boundary of $f(T)$ onto the boundary of the standard simplex in $\R^3$. The composition $g\circ f$ verifies the claim of our Theorem.
\end{proof}

\pagebreak

\bibliographystyle{plain}
\bibliography{Poincare}

\end{document}